\documentclass[psamsfonts, dvipsnames, 10pt]{amsart}
\usepackage{xcolor}
\usepackage[colorlinks=true,
pdfstartview=FitV, linkcolor=red, citecolor=green,
urlcolor=blue]{hyperref}
 
\usepackage{amsmath,amsfonts,latexsym,amssymb}
\usepackage{amssymb,amsfonts, amsmath}
\usepackage{mathrsfs}
\usepackage[latin1]{inputenc}
\usepackage[T1]{fontenc}
\usepackage{ae,aecompl}
\usepackage{braket}
\usepackage{comment}
\usepackage{amssymb,amsfonts, amsmath}
 \usepackage{mathtools} 
\usepackage[all,arc]{xy}
\usepackage{enumerate}
\usepackage{mathrsfs}
\usepackage{color}
\usepackage{graphicx}
\usepackage{subfig}
\usepackage{verbatim} 
\usepackage{amssymb,amsfonts, amsmath}
\usepackage[all,arc]{xy}
\usepackage{enumerate}
\usepackage{mathrsfs}
\usepackage{hyperref}
\usepackage{xstring}

\newtheorem{theorem}{Theorem}[section]
\newtheorem{lemma}[theorem]{Lemma}
\newtheorem{proposition}[theorem]{Proposition}

\newtheorem{definition}[theorem]{Definition}

\newtheorem{observation}[theorem]{Observation}
\renewcommand{\leq}{\leqslant}
\renewcommand{\geq}{\geqslant}

\long\def\@savemarbox#1#2{\global\setbox#1\vtop{\hsize\marginparwidth 
  \@parboxrestore\tiny\raggedright #2}}
\title{Quasi-isometric embeddings inapproximable by Anosov representations}
\author{Konstantinos Tsouvalas}
\date{\today}
\AtEndDocument{\bigskip{\footnotesize%
  \textsc{Department of Mathematics, University of Michigan, 530 Church
Street, Ann Arbor, MI 48109, USA} \par  
 \textit{E-mail address}: \texttt{tsouvkon@umich.edu}}}

\begin{document}
\maketitle
\begin{abstract}{We construct examples of quasi-isometric embeddings of word hyperbolic groups into $\mathsf{SL}(d,\mathbb{R})$ for $d \geqslant 5$ which are not limits of Anosov representations into $\mathsf{SL}(d,\mathbb{R})$. As a consequence, we conclude that an analogue of the density theorem for $\mathsf{PSL}(2,\mathbb{C})$ does not hold for $\mathsf{SL}(d,\mathbb{R})$ when $d \geq 5$.}\end{abstract}
\section{Introduction}
In this short paper we construct the first examples of word hyperbolic and quasi-isometrically embedded subgroups $\Gamma$ of $\mathsf{SL}(d,\mathbb{R})$ for $d\geqslant 5$ which are not limits of Anosov representations of $\Gamma$ into $\mathsf{SL}(d,\mathbb{R})$. More precisely, we prove:

\begin{theorem} \label{limit} Let $g \geqslant 1$ and $\Gamma$ be the word hyperbolic group with presentation  $$\Bigg \langle a_1,b_1,...,a_{2g},b_{2g}, c_{1},d_1,...,c_{2g},d_{2g} \Bigg| \begin{matrix}
[a_1,b_1]...[a_{2g},b_{2g}], [c_1,d_1]...[c_{2g},d_{2g}],\\
[a_1,b_1]...[a_g,b_g][c_1,d_1]...[c_{g},d_{g}] 

\end{matrix} \Bigg \rangle$$ 
\noindent \textup{(i)} For every $d \geqslant 5$ there exists a quasi-isometric embedding $\rho:\Gamma \rightarrow \mathsf{SL}(d,\mathbb{R})$ such that $\rho$ is not a limit of Anosov representations of $\Gamma$ into $\mathsf{SL}(d,\mathbb{R})$.\\

\noindent \textup{(ii)} For $g \geqslant 4$, there exists a strongly irreducible, quasi-isometric embedding \hbox{$\rho:\Gamma \rightarrow \mathsf{SL}(12,\mathbb{R})$} such that $\rho$ is not a limit of Anosov representations of $\Gamma$ into $\mathsf{SL}(12, \mathbb{R})$.  \end{theorem}

The domain group of the representation $\rho$ in Theorem \ref{limit} is the fundamental group of a book of I-bundles and by Thurston's Geometrization theorem (see Morgan \cite{Morgan}) it admits a convex cocompact representation into $\mathsf{PSL}(2,\mathbb{C})$. We remark that for $d \geqslant 6$ in Theorem \ref{limit} we can replace $\Gamma$ with any word hyperbolic $3$-manifold group which retracts to a free subgroup of rank at least $8$ and is not virtually a free or a surface group. The density conjecture for Kleinian groups established by the work of Brock-Bromberg \cite{density}, Brock-Canary-Minsky \cite{endinglamination}, Namazi-Souto \cite{Namazi-Souto} and Ohshika \cite{Ohshika} implies that every discrete and faithful representation of a word hyperbolic group into $\mathsf{PSL}(2,\mathbb{C})$ is an algebraic limit of Anosov representations. The examples of Theorem \ref{limit} demonstrate the failure of the density conjecture for the higher rank Lie group $\mathsf{SL}(d,\mathbb{R})$ for $d \geqslant 5$. \par In Theorem \ref{limit2} we produce examples similar to those in Theorem \ref{limit} (ii) in infinitely many dimensions all of whose elements are semiproximal. Moreover, in Proposition \ref{surface} we also provide examples of quasi-isometric embeddings of surface and free groups into $\mathsf{SL}(4,\mathbb{R})$ and $\mathsf{SL}(6,\mathbb{R})$ which are not in the closure of the space of Anosov representations.  \par An example of a quasi-isometric embedding of the free group of rank $2$ which is not Anosov was constructed by Guichard in \cite{Guichard}, see also \cite[Proposition A.1, p. 67]{GGKW}. Moreover, Guichard's example is unstable, i.e. it is a limit of indiscrete representations but also a limit of $P_2$-Anosov representations (see Definition \ref{Anosov}) of the free group of rank 2 into $\mathsf{SL}(4,\mathbb{R})$. \par For our constructions we shall use the following fact: for a $P_1$-Anosov subgroup $\Gamma$ of $\mathsf{SL}(d,\mathbb{R})$, $d \geqslant 4$, every quasiconvex infinite index subgroup $\Delta$ of $\Gamma$  with connected Gromov boundary contains a finite-index subgroup all of whose elements are positively semiproximal (see Lemma \ref{positive1}). It follows that if $\rho:\Gamma \rightarrow \mathsf{SL}(k,\mathbb{R})$ is a limit of $P_i$-Anosov representations, or equivalently $\wedge^i \rho$ is a limit of $P_1$-Anosov representations, then the group $\wedge^{i}\rho(\Delta)$ contains a finite index subgroup consisting entirely of positively semiproximal elements. We remark that the connecteness of the Gromov boundary of $\Delta$ cannot be dropped since every non-cyclic free subgroup of a lift of a Hitchin surface group into $\mathsf{SL}(2d,\mathbb{R})$ contains an element all of whose eigenvalues are negative.

\vspace{0.3cm}

\noindent \textbf{Acknowledgements.} I would like to thank my advisor Richard Canary for his support and encouragement and Sara Maloni for motivating discussions which led to the construction of the examples of this paper. I also thank Jeff Danciger, Fran\c cois Labourie and Nicolas Tholozan for helpful discussions and comments. This work was partially supported by grants DMS-1564362 and DMS-1906441 from the National Science Foundation.

\section{Background}
In this section we define Anosov representations and prove two lemmas required for our construction. For a transformation $g \in \mathsf{SL}(d,\mathbb{R})$ we denote by $\ell_{1}(g) \geqslant ...\geqslant \ell_{d}(g)$ and $\sigma_1(g)\geqslant ...\geqslant \sigma_{d}(g)$ the moduli of the eigenvalues  and singular values of $g$ in decreasing order respectively. We recall that $\sigma_{i}(g)=\sqrt{\ell_{i}(gg^t)}$ for $1\leqslant i \leqslant d$, where $g^{t}$ denotes the transpose matrix. For $1 \leqslant i \leqslant d-1$, the matrix $g \in \mathsf{SL}(d,\mathbb{R})$ is called $P_i$-{\em proximal} if $\ell_{i}(g)>\ell_{i+1}(g)$. In the case $g$ is $P_1$-proximal we denote by $\lambda_{1}(g)$ the unique eigenvalue of $g$ of maximum modulus. The matrix $g$ is called \emph{semiproximal} if either $\ell_1(g)$ or $-\ell_1(g)$ is an eigenvalue of $g$ and \emph{positively semiproximal} if $\ell_1(g)$ is an eigenvalue of $g$.

\subsection{Anosov representations} For a finitely generated group $\Gamma$ we fix a left invariant word metric $d_{\Gamma}$ and for $\gamma \in \Gamma$ $|\gamma|_{\Gamma}$ denotes the distance of $\gamma$ from the identity element $e \in \Gamma$. If $\Gamma$ is word hyperbolic, $\partial_{\infty}\Gamma$ denotes the Gromov boundary of $\Gamma$. A representation $\rho:\Gamma \rightarrow \mathsf{SL}(d,\mathbb{R})$ is called a {\em quasi-isometric embedding} if  the orbit map of $\rho$, $\tau_{\rho}:\Gamma \rightarrow \mathsf{SL}(d,\mathbb{R})/\mathsf{SO}(d)$ where $\tau_{\rho}(\gamma)=\rho(\gamma)\mathsf{SO}(d)$ for $\gamma \in \Gamma$ is a quasi-isometric embedding. In other words, there exist constants $J,K>0$ such that $$\frac{1}{K}e^{\frac{1}{J}|\gamma|_{\Gamma}} \leqslant \frac{\sigma_1(\rho(\gamma))}{\sigma_{d}(\rho(\gamma))} \leqslant Ke^{J|\gamma|_{\Gamma}}$$ for every $\gamma \in \Gamma$. For a representation of a finitely generated group a much stronger property than being a quasi-isometric embedding is to be {\em Anosov}. Anosov representations were introduced by Labourie in \cite{labourie-invent} in his study of Hitchin representations and further developed by Guichard-Wienhard in \cite{GW}. We define Anosov representations by using a characterization in terms of gaps between singular values of elements, established by Kapovich-Leeb-Porti in \cite{KLP1} and Bochi-Potrie-Sambarino \cite{BPS}. 

\begin{definition} \label{Anosov} Let $\Gamma$ be a finitely generated group and $\rho:\Gamma \rightarrow \mathsf{SL}(d,\mathbb{R})$ be a representation. For $1 \leqslant i \leqslant \frac{d}{2}$, the representation $\rho$ is called $P_{i}$-Anosov if there exist constants $C,a>0$ with the property: $$\frac{\sigma_i(\rho(\gamma))}{\sigma_{i+1}(\rho(\gamma))} \geqslant Ce^{a|\gamma|_{\Gamma}}$$ for every $\gamma \in \Gamma$.  \end{definition} 
In addition, it was proved in \cite{KLP1} and \cite{BPS} that a finitely generated group which admits an Anosov representation into $\mathsf{SL}(d,\mathbb{R})$ is necessarily word hyperbolic. We call a representation $\rho:\Gamma \rightarrow \mathsf{SL}(d,\mathbb{R})$ Anosov if it is $P_i$-Anosov for some $i$. Note that $\rho$ is $P_i$-Anosov if and only if the exterior power $\wedge^i\rho$ is $P_1$-Anosov. The property of being Anosov is stable, i.e. for every $P_i$-Anosov representation $\rho$ there exists an open neighbourhood $U$ of $\rho$ in $\textup{Hom}(\Gamma, \mathsf{SL}(d,\mathbb{R}))$ consisting entirely of $P_i$-Anosov (see \cite{labourie-invent} and \cite[Theorem 5.14]{GW}). Examples of Anosov representations include quasi-isometrically embedded subgroups of simple real rank $1$ Lie groups and their small deformations into higher rank Lie groups, Hitchin representations and holonomies of strictly convex projective structures on closed manifolds. \par For $1\leqslant m \leqslant d-1$, denote by $\mathsf{Gr}_{m}(\mathbb{R}^d)$ the Grassmannian of $m$-planes in $\mathbb{R}^d$. Every $P_{i}$-Anosov representation $\rho:\Gamma \rightarrow \mathsf{SL}(d,\mathbb{R})$ admits a pair of continuous, $\rho$-equivariant maps \hbox{$\xi_{\rho}^{i}:\partial_{\infty}\Gamma \rightarrow \mathsf{Gr}_{i}(\mathbb{R}^d)$} and \hbox{$\xi_{\rho}^{d-i}:\partial_{\infty}\Gamma \rightarrow \mathsf{Gr}_{d-i}(\mathbb{R}^d)$} called the {\em Anosov limit maps}. We refer the reader to \cite{GW} and \cite{GGKW} for a careful discussion of Anosov limit maps and their properties. We mention here some of their main properties:
\medskip

\noindent \textup{(i)} The maps $\xi_{\rho}^{i}$ and $\xi_{\rho}^{d-i}$ are {\em compatible}, i.e. for every $x \in \partial_{\infty}\Gamma$, $\xi_{\rho}^{k}(x) \subset \xi_{\rho}^{d-i}(x)$.\\
\noindent \textup{(ii)} For every $\gamma \in \Gamma$ non-trivial, $\rho(\gamma)$ is $P_{i}$ and $P_{d-i}$-proximal where $\xi_{\rho}^{i}(\gamma^{+})$ and $\xi_{\rho}^{d-i}(\gamma^{+})$ are the attracting fixed points of $\rho(\gamma)$ in $\mathsf{Gr}_{i}(\mathbb{R}^d)$ and $\mathsf{Gr}_{d-i}(\mathbb{R}^d)$ respectively. \\
\noindent \textup{(iii)} The maps $\xi_{\rho}^{i}$ and $\xi_{\rho}^{d-i}$ are {\em transverse}, i.e. for every $x,y \in \partial_{\infty}\Gamma$ with $x \neq y$, $\mathbb{R}^d=\xi_{\rho}^{i}(x)\oplus \xi_{\rho}^{d-i}(y)$.\\

\par The key property of Anosov representations that we use for our construction is that when $\Gamma$ is neither a free or a surface group, then for a $P_1$-Anosov representation $\rho:\Gamma \rightarrow \mathsf{SL}(d,\mathbb{R})$, the image $\rho(\Gamma)$ contains many elements with positive first eigenvalue. The following lemma is essential for the construction of our examples. For a finitely generated group $\Gamma$ we denote by $\Gamma_2$ the intersection of all finite-index subgroups of $\Gamma$ of index at most $2$. An open subset $\Omega$ of $\mathbb{P}(\mathbb{R}^d)$ is called {\em properly convex} if it is bounded and convex in an affine chart of $\mathbb{P}(\mathbb{R}^d)$.
\medskip

\begin{lemma}\label{positive1} Let $\Gamma$ be a word hyperbolic group and $\Delta$ a quasiconvex and infinite index subgroup of $\Gamma$ such that $\partial_{\infty} \Delta$ is connected. Let $d \geqslant 4$ and $\rho:\Gamma \rightarrow \mathsf{SL}(d,\mathbb{R})$ be a representation. Suppose that there exists a sequence of $P_1$-Anosov representations $\big \{\rho_{n}:\Gamma \rightarrow \mathsf{SL}(d,\mathbb{R})\big \}_{n \in \mathbb{N}}$ with $\lim_{n} \rho_n=\rho$. Then every $\delta \in \Delta_2$ is positively semiproximal.\end{lemma}

\begin{proof} Suppose that $\rho_{0}$ is a $P_1$-Anosov representation with Anosov limit map \hbox{$\xi_{\rho_0}^{1}:\partial_{\infty}\Gamma \rightarrow \mathbb{P}(\mathbb{R}^d)$}. Then, since $\Delta$ has infinite index and is quasiconvex in $\Gamma$, we may find $\delta_0 \in \Gamma$ with $\delta_{0}^{\pm} \notin \partial_{\infty}\Delta$. By transversality, the affine chart $A \subset \mathbb{P}(\mathbb{R}^d)$ defined by the $(d-1)$-plane $\xi_{\rho_0}^{d-1}(\delta_{0}^{+})$ contains the connected compact set $\mathcal{C}_{\Delta}=\xi_{\rho_0}^{1}(\partial_{\infty}\Delta)$. Let $V=\big \langle u: [u] \in \mathcal{C}_{\Delta} \big \rangle$. Note that $\rho_0(\Delta)$ preserves $\mathcal{C}_{\Delta}$ and the restriction of $\rho_{0}|_{\Delta}$ on $V$ is \hbox{$P_1$-Anosov} since $\mathcal{C}_{\Delta}$ contains the attracting fixed point of $\rho_0(\delta)$ for every $\delta  \in \Delta$. The connected set $\mathcal{C}_{\Delta}$ also lies in the affine chart $A\cap \mathbb{P}(V)$ of $\mathbb{P}(V)$ and the convex hull $\textup{Conv}_{A\cap \mathbb{P}(V)}(\mathcal{C}_{\Delta})$ is preserved by $\rho_{0}|_{V}(\Delta)$. In particular, the properly convex subset $\textup{Int}\big(\textup{Conv}_{A\cap \mathbb{P}(V)}(\mathcal{C}_{\Delta})\big)$ is preserved by $\rho_{0}|_{V}(\Delta)$. We obtain a representation $\widetilde{\rho_0}:\Delta \rightarrow \mathsf{GL}(V)$ whose image preserves a properly convex open cone $C\subset V$ and $\widetilde{\rho_0}(\delta)=\rho_0|_{V}(\delta)$ for every $\delta \in \Delta_2$. Note that the attracting fixed point of $\delta \in \Delta_2$ is always in $V$ and $\lambda_1(\widetilde{\rho}(\delta))=\lambda_1(\rho(\delta))=\lambda_1(\rho|_{V}(\delta))$. By \cite[Lemma 3.2]{benoist-divisible3} we have $\lambda_1(\rho_0(\delta))>0$ and hence $\rho_0(\delta)$ is positively proximal. Now let $\delta \in \Delta_2$. By the previous arguments, for every $n \in \mathbb{N}$, $\lambda_1(\rho_n(\delta))>0$ and there exists unit vector $u_n \in \mathbb{R}^d$ such that \hbox{$\rho_n(\delta)u_n=\lambda_1(\rho_n(\delta))u_n$}. Up to passing to a subsequence, we may assume that $\lim_{n}\lambda_1(\rho_n(\delta))$ exists and has to be an eigenvalue (not necessarily unique) of $\lim_{n}\rho_n(\delta)=\rho(\delta)$ of maximum modulus. The conclusion follows.  \end{proof}

On the other hand, the image of Anosov representations might contain elements which are not positively proximal. In fact, this is the case for Fuchsian representations into $\mathsf{SL}(2,\mathbb{R})$.
\medskip
\begin{lemma} \label{negative} Let $F_k$ denote the free group on $k \geqslant 2$ generators. Let  $j: F_k \rightarrow \mathsf{SL}(2,\mathbb{R})$ be a quasi-isometric embedding and $F$ be a free subgroup of $F_k$ of rank at least $2$. Then for every $a \in F_k$ with $a \notin F$, there exists $w \in [F,F]$ such that $\lambda_1 (j(wa))<0$.\end{lemma}

\begin{proof} Note that $j([F,F])$ is discrete in $\mathsf{SL}(2,\mathbb{R})$, hence by \cite[Lemma 2]{CG} (see also \cite[Theorem 1.6]{benoist-divisible0}), there exists $w_{0} \in [F,F]$ such that $\lambda_1(j(w_0))<0$. Then, we may write $$j(w_0)=h \begin{bmatrix}
\lambda_1(j(w_0)) & 0\\ 
 0 & \frac{1}{\lambda_1(j(w_0))}
\end{bmatrix}h^{-1}$$ Since \hbox{$\{w_{0}^{+},w_{0}^{-} \}\cap \{a^{+}, a^{-}\}$} is empty and $j$ is $P_1$-Anosov, by transversality we have that the line \hbox{$j(w_0)\xi_{1}^{j}(a^{\pm})=\xi_{1}^{j}(w_{0}a^{\pm})$} is different from $\xi_{1}^{j}(a^{+}),\xi_{1}^{j}(a^{-})$ and hence $\big \langle h^{-1}j(a)he_1, e_1\big \rangle$ is not zero. Then we notice that $$\lim_{n \rightarrow \infty}\frac{\lambda_{1}(j((w_{0}^{n}a))}{\lambda_1(j(w_{0}^n))}= \big \langle h^{-1}j(a)he_1,e_1 \big \rangle \  \textup{and} \ \lim_{n \rightarrow \infty}\frac{\lambda_1(j(w_0^{2n+1}a))}{\lambda_1(j(w_{0}^{2n}a))}=\lambda_{1}(j(w_0))<0$$  For large enough $n \in \mathbb{N}$, the numbers $\lambda_{1}(j(w_{0}^{2n}a))$ and $\lambda_{1}(j(w_{0}^{2n+1}a))$ have opposite signs and the conclusion follows.\end{proof}

We also need the following observation:

\begin{observation} \label{free} Let $F_2$ be the free group on $\{a,b\}$ and $\rho:F_2 \rightarrow \mathsf{SL}(2,\mathbb{R})$ be a quasi-isometric embedding. Let $k\in \mathbb{N}$ and $\phi_{k}:F_2 \rightarrow F_2$ be the monomorphism defined by $\phi_{k}(a)=b^kab^k$ and $\phi_k(b)=a^kba^k$. There exists $C>0$ such that $\ell_1(\rho(\phi_k(\gamma)) \geqslant \ell_1(\rho(\gamma))^{Ck}$ for every $\gamma \in F_2$ and $k \in \mathbb{N}$. \end{observation}
\medskip

\section{The construction}
By using Lemma \ref{positive1} and \ref{negative} we construct representations of the fundamental group $\Gamma$ of a book of I-bundles which are not limits of Anosov representations of $\Gamma$ in $\mathsf{SL}(d,\mathbb{R})$ for $d \geq 5$. We recall that given a group $K$ and a subgroup $H$ of $H$, a homomorphism $r:K \rightarrow H$ is called a {\em retract} if $r(h)=h$ for every $h \in H$.
\medskip

\noindent {\em Proof of Theorem \ref{limit}.} Let $\Delta =\langle a_1,b_1,...,a_{2g},b_{2g} \rangle$ and \hbox{$F=\langle a_{1},b_{1},...,a_{g},b_{g}\rangle$} be the subgroup of $\Delta$ which is free on $2g$ generators. Note that there exists a retract of $\Gamma$ onto the surface subgroup $\Delta$. Then $\Delta$ retracts onto $F$, by sending $a_{i} \mapsto a_{i}, b_{i}\mapsto b_{i}$, $a_{g+i} \mapsto b_{g-i+1}$ and $b_{g+i}\mapsto a_{g-i+1}$ for $1 \leqslant i \leqslant g$. We finally obtain a retract $R:\Gamma \rightarrow F$.
\par We first construct reducible examples in all dimensions greater than or equal to $5$. By \cite{CMT} (see Section 4 page 26), there exists a convex cocompact representation $i:\Gamma \xhookrightarrow{} \mathsf{SL}(2,\mathbb{C})$ such that $\Delta$ is a subgroup of $\mathsf{SL}(2,\mathbb{R})$. Let $S: \mathsf{SL}(2,\mathbb{C}) \rightarrow \mathsf{SO}_{0}(3,1)$ be the spin homomorphism such that $S(\textup{diag}(a,\frac{1}{a}))$ is conjugate to $\textup{diag}\big(a^2,1,1,1, \frac{1}{a^2}\big)$ for every $a \in \mathbb{R}$. By composing $i$ with $S$ we obtain a quasi-isometric embedding \hbox{$\rho_{0}:\Gamma \rightarrow \mathsf{SO}(3,1)$} which is a $P_1$-Anosov representation regarded as a representation into $\mathsf{SL}(4,\mathbb{R})$. Note that for every $\gamma \in \Delta$, the matrix $\rho_0(\gamma)$ is positively proximal.
\medskip

\textup{(i)} Suppose that $d=5$. By postcomposing $i$ with the irreducible representation $\tau_2:\mathsf{SL}(2,\mathbb{C}) \rightarrow \mathsf{SL}(4,\mathbb{R})$, where $$\tau_2(g)=\begin{bmatrix}
\textup{Re}(g) & -\textup{Im}(g) \\ 
\textup{Im}(g) & \textup{Re}(g)
\end{bmatrix}, \ g \in \mathsf{SL}(2,\mathbb{C}),$$ we obtain a $P_2$-Anosov representation \hbox{$\rho_1:\Gamma \rightarrow \mathsf{SL}(4,\mathbb{R})$} such that $\rho_{1}(\Delta)$ is a subgroup of $\tau_2(\mathsf{SL}(2,\mathbb{R}))$. By Lemma \ref{negative}, we can find \hbox{$w \in [F,F] \subset \Delta_{2}$} such that \hbox{$\lambda_{1}(\varphi(wa_{1}^2))<0$}. Now we consider a non-trivial map \hbox{$\varepsilon: F \rightarrow \mathbb{R}^{+}$ }such that $\varepsilon(wa_{1}^2)=\varepsilon(a_1^2)=x$ with $x^{5/4}> \ell_{1}(\rho(wa_1^2))$. We consider the representation \hbox{$\rho:\Gamma \rightarrow \mathsf{SL}(5,\mathbb{R})$} defined as: $$\rho(\gamma)=\begin{bmatrix}
\frac{1}{\sqrt[4]{\varepsilon(\gamma)}} \rho_{1}(\gamma) & 0 \\ 
   0  & \varepsilon(R(\gamma))\\
\end{bmatrix}, \ \gamma \in \Gamma$$ Notice that the first $3$ eigenvalues of the matrix $\rho(wa_1^2)$ are $x$, $x^{-1/4}\lambda_1(\rho_1(wa_1^2))$ and $x^{-1/4}\lambda_1(\rho_1(wa_1^2))$. The matrix $\wedge^2 \rho(wa_1^2)$ is not positively semiproximal. Since $wa_1^2 \in \Delta_2$, by Lemma \ref{positive1} the representation $\rho$ cannot be a limit of $P_2$-Anosov representations of $\Gamma$ into $\mathsf{SL}(4,\mathbb{R})$. Note also that $\textup{ker}(\varepsilon)\cap \Delta_2$ contains a free subgroup and $\varphi(\textup{ker}(\varepsilon) \cap \Delta_2)$ is discrete subgroup of $\mathsf{SL}(2,\mathbb{R})$. Hence, by Lemma \ref{negative}, there exists $h \in \Delta_2$ with $\varepsilon(h)=1$ and \hbox{$\lambda_1(\rho_1(h))=\lambda_1(\varphi(h))<0$}. Therefore, by Lemma \ref{positive1} $\rho$ is not a limit of $P_1$-Anosov representations of $\Gamma$ into $\mathsf{SL}(5,\mathbb{R})$.
\par We now assume that $d=6$. We can find a quasi-isometric embedding \hbox{$j:F \rightarrow \mathsf{SL}(2,\mathbb{R})$} such that \hbox{$\ell_{1}(j(\gamma))\geqslant \ell_{1}(\rho_{0}(\gamma))^2$} for every $\gamma \in F$. Now we consider the representation \hbox{$\rho:\Gamma \rightarrow \mathsf{SL}(6,\mathbb{R})$} defined as follows: $$\rho(\gamma)=\begin{bmatrix}
\rho_{0}(\gamma) & 0 \\ 
  0  & j(R(\gamma))
\end{bmatrix}, \ \gamma \in \Gamma$$ By Lemma \ref{negative}, we can find $w \in F \cap \Delta_2$ such that $\lambda_{1}(j(w))<0$. Since $\ell_{1}(j(w))>\ell_{1}(\rho_{0}(w))$, the matrix $\rho(w)$ is both $P_{1}$ and $P_{2}$-proximal, $\lambda_{1}(\rho(w))<0$ and $\lambda_{1}(\wedge^2 \rho(w))<0$. Moreover, the matrix $\wedge^3 \rho(w)$ has the number $\lambda_1(j(w)) \ell_1(\rho_0(w))<0$ as an eigenvalue of maximum modulus and multiplicity two. It follows by the previous lemma that $\rho, \wedge^2 \rho$ and $\wedge^3 \rho$ cannot be a limit of $P_{1}$-Anosov representations. This completes the proof of this case.
\par Now suppose $d \geqslant 7$. We consider again a convex cocompact representation \hbox{$j: \langle a_1,b_1 \rangle \rightarrow \mathsf{SL}(2,\mathbb{R})$} which uniformly dominates the restriction $\rho_{0}|_{\langle a_1,b_1\rangle}$. There exists $w \in [\langle a_1,b_1\rangle, \langle a_1,b_1\rangle]$ such that $\lambda_1(j(\iota_1'(wa_1^2))<0$. We consider group homomorphisms $\varepsilon_{1},...,\varepsilon_{d-6}:\langle a_1,b_1 \rangle \rightarrow \mathbb{R}^{+}$ such that $\ell_1(j(wa_1^2))>\varepsilon_1(a_1^2)>...>\varepsilon_{d-6}(a_1^2)>\ell_1(\rho_{0}(wa_1^2))$. The representation $\rho:\Gamma \rightarrow \mathsf{GL}(d,\mathbb{R})$ defined by the blocks $$\rho(\gamma)=\textup{diag} \Big(\rho_{0}(\gamma), j\big(R(\gamma)\big), \varepsilon_1\big(R(\gamma)\big),...,\varepsilon_{d-6}\big(R(\gamma)\big) \Big)$$ has the property that $\wedge^{i} \rho(wa_1^2)$ is proximal but not positively proximal for $i=1,...,d-4$. Finally, by Lemma \ref{positive1} it follows that for every $1 \leqslant i \leqslant \frac{d}{2}$, the representation $\frac{1}{\sqrt[d]{\textup{det}(\rho(\gamma))}}\rho(\gamma)$ is not a limit of $P_{i}$-Anosov representations.
\par \textup{(ii)} Now we construct a strongly irreducible representation of $\Gamma$ into $\mathsf{SL}(12,\mathbb{R})$ which is not a limit of $P_i$-Anosov representations for $1 \leqslant i \leqslant 6$. We assume that $g \geqslant 4$. By Observation \ref{free}, we can find quasi-isometric embeddings \hbox{$\iota_1: \langle a_1,b_1,a_2 \rangle \rightarrow \mathsf{SL}(2,\mathbb{R})$} and \hbox{$\iota_2: \langle b_2,a_3,b_3 \rangle  \rightarrow \mathsf{SL}(2,\mathbb{R})$} such that \hbox{$\ell_1(\iota_1(g)) \geqslant \ell_1(S(g))^{3}$}, for every $g \in \langle a_1,b_1,a_2 \rangle$ and $\ell_1(\iota_2(h)) \geqslant \ell_1(S(h))^{5}$ for every $h \in \langle b_2,a_3,b_3 \rangle$. By Lemma \ref{negative}, we can find \hbox{$w \in \Gamma_2 \cap \langle a_1,b_1 \rangle$} such that \hbox{$\lambda:=\lambda_1(\iota_1(wa_{2}^2))<0$}. Let $\mu=\lambda_1(S(wa_2^2))>0$. Now consider a non-trivial map \hbox{$\varepsilon: \langle a_1,b_1,a_2 \rangle \rightarrow \mathbb{R}^{+}$} such that $\varepsilon(a_2)=x$ and $$|\lambda|>x^{3}> \frac{|\lambda|}{\mu^2}>1>\frac{\mu^2}{|\lambda|}$$ and the representation $\iota_{1}': \langle a_1,b_1, a_2 \rangle \rightarrow \mathsf{SL}(3,\mathbb{R})$ defined as follows: $$\iota_1'(\gamma)= \begin{bmatrix}
\frac{1}{\sqrt{\varepsilon(\gamma)}}\iota_1(\gamma) & 0 \\ 
  0   & \varepsilon(\gamma) \\
\end{bmatrix}, \ \gamma \in \langle a_1,b_1, a_2 \rangle $$ 
Notice that $\varepsilon(wa_2^2)=x^2$ and by the choice of $x$, the matrix $\iota_1'(wa_2^2)$ is proximal with eigenvalues in decreasing order $\frac{\lambda}{x},x^2,\frac{1}{\lambda x}$. By Lemma \ref{negative} we can also find $z \in [\langle b_2, a_3\rangle, \langle b_2,a_3 \rangle]$ such that $s:=\lambda_{1}(\iota_2(zb_3^2))<0$. We consider the representations $\iota_{2}': \langle b_2,a_3,b_3 \rangle \rightarrow \mathsf{SL}(3,\mathbb{R})$ defined as follows: $$\iota_2'(\delta)= \begin{bmatrix}
\iota_2(\delta) & 0 \\ 
  0   & 1 \\
\end{bmatrix}, \ \delta \in \langle b_2,a_3,b_3 \rangle $$ and $A:F \rightarrow \mathsf{SL}(3,\mathbb{R})$ defined as follows: $$A(\gamma)=\iota_1'(\gamma), \ \gamma \in \langle a_1,b_1,a_2 \rangle$$ $$A(\delta)=\iota_2'(\delta), \ \delta \in \langle b_2,a_3,b_3 \rangle$$ $$A(\langle  a_3,b_4,...,a_{g},b_{g} \rangle) \ \textup{is chosen to be Zariski dense in} \ \mathsf{SL}(3,\mathbb{R})$$ We obtain a Zariski dense representation $R \circ A:\Gamma \rightarrow \mathsf{SL}(3,\mathbb{R})$. 
\par We first observe that $\rho_0 \otimes (A \circ R)$ is is strongly irreducible and a quasi-isometric embedding. For every finite-index subgroup $H$ of $\Gamma $ the restriction of the product \hbox{$\rho_0 \times (A\circ R):H \rightarrow \mathsf{SO}(3,1) \times \mathsf{SL}(3,\mathbb{R})$} is Zariski dense. Note that the tensor product representation $\otimes: \mathsf{SO}(3,1) \times \mathsf{SL}(3,\mathbb{R}) \rightarrow \mathsf{SL}(12,\mathbb{R})$, $(g,h) \mapsto g \otimes h$ is irreducible. Hence, any proper $(\rho_0 \otimes (A\circ R))(H)$-invariant subspace $V$ of \hbox{$\mathbb{R}^{12}=\mathbb{R}^4 \otimes_{\mathbb{R}} \mathbb{R}^3$} has to be invariant under $g \otimes h$ for every $g \in \mathsf{SO}(3,1)$ and $h \in \mathsf{SL}(3,\mathbb{R})$. Therefore, $V$ is trivial and $\rho_0 \otimes (A \circ R)$ is strongly irreducible. Moreover, for every $\gamma \in \Gamma$, \hbox{$\sigma_1\big(\rho_0(\gamma)\otimes A(R(\gamma))\big)=$} $\sigma_1(\rho_0(\gamma))\sigma_1(A(R(\gamma)))$ and hence $\rho_0 \otimes (A\circ R)$ is a quasi-isometric embedding since $\rho_0$ is \hbox{$P_1$-Anosov.}
\par We claim that the tensor product representation $\rho_0 \otimes (A\circ R):\Gamma \rightarrow \mathsf{SL}(12,\mathbb{R})$ is not a limit of Anosov representations. We consider the element $wa_{2}^2 \in \Gamma_2$. We have $A(R(wa_2^2))=A(wa_2^2)=\iota_1'(wa_2^2)$ and the matrix \hbox{$\rho_0(wa_{2}^2)\otimes A(wa_2^2)$} is conjugate to $$g=\ \textup{diag} \Big(\mu^2, 1,1, \frac{1}{\mu^2} \Big) \otimes \textup{diag} \Big(\frac{\lambda}{x}, x^2, \frac{1}{\lambda x} \Big), \ \ \lambda=\lambda_1(wa_2^2)<0$$ By the choice of $x$, since $|\lambda|>x^3>\frac{|\lambda|}{\mu^2}>1$, the first $7$ eigenvalues in decreasing order of their moduli are $$\frac{\lambda}{x}\mu^2, \ x^2 \mu^2,\  \frac{\lambda}{x},\ \frac{\lambda}{x},\ x^2,\ x^2, \ \frac{\lambda}{x\mu^2} $$ The matrix $\wedge^{i}g$ is proximal for $i=1,2,4,6$ but not positively proximal. Thus, by Lemma \ref{positive1}, $\rho \otimes A$ is not a limit of $P_i$-Anosov representations for $i=1,2,4,6$. The matrix $\wedge^5 g$ has as an eigenvalue of maximum modulus and multiplicity $2$ the number $\lambda^3 \mu^4 x<0$. Therefore, $\wedge^5 g$ is not positively semiproximal and $\rho \otimes A$ cannot be a limit of $P_5$-Anosov representations, again by Lemma \ref{positive1}. Now we consider the element $zb_{3}^2 \in \Gamma_2$. Note that $R(A(zb_3^2))=A(zb_3^2)=\iota_2'(zb_3^2)$ and \hbox{$\rho_0(zb_{3}^2)\otimes A(zb_3^2)$} is conjugate to the matrix $$h=\textup{diag} \Big( \nu^2,1,1, \frac{1}{\nu^2}\Big) \otimes \textup{diag} \Big(s,1, \frac{1}{s} \Big), \ s=\lambda_1(\iota_2(zb_3^2))<0, \ \nu=\lambda_1(S(zb_3^2))$$ Since $|s|>\nu^4$, the first $5$ eigenvalues of $h$ in decreasing order of their moduli are $$s \nu^2, \ s, \ s, \ \frac{s}{\nu^2}, \nu^2$$ We notice that $\wedge^3 h$ is proximal with first eigenvalue $s^3 \nu^2<0$. It follows that $\rho \otimes A$ is not a limit of $P_3$-Anosov representations. $\qed$\\

\begin{rmk}  The proof of Lemma \ref{positive1} only uses the fact that every $P_i$-Anosov representation admits a continuous, dynamics preserving and $\rho$-equivariant map $\xi_{\rho}^{i}:\partial_{\infty}\Gamma \rightarrow \mathsf{Gr}_{i}(\mathbb{R}^d)$. Therefore, in Theorem \ref{limit} the representation $\rho$ admits an open neighbourhood $U \subset \textup{Hom}(\Gamma, \mathsf{SL}(d,\mathbb{R}))$ such that for every $1 \leqslant i\leqslant \frac{d}{2}$ and every $\rho' \in U$ there is no continuous, $\rho'$-equivariant and dynamics preserving map \hbox{$\xi_{\rho'}^{i}:\partial_{\infty}\Gamma \rightarrow \mathsf{Gr}_{i}(\mathbb{R}^d)$}. \end{rmk}

\section{Additional examples}
By following similar arguments as in Theorem \ref{limit} (ii) and by increasing the number of surface vertex groups of the fundamental group of the I-bundle, it is possible to obtain strongly irreducible quasi-isometric embeddings for infinitely many odd dimensions such that every non-trivial element is semiproximal. We denote by $S: \mathsf{SL}(2,\mathbb{C})\rightarrow \mathsf{SO}_{0}(3,1)$ the spin homomorphism.
\begin{theorem} \label{limit2} Let $g\geqslant 3$ and $\Gamma$ be as in Theorem \ref{limit}, $n \geqslant 5$ and let $\Delta_n$ be the amalgamated free product of $2^{n-3}+1$ copies of $\Gamma$ along the maximal cyclic subgroup of $\Gamma$ generated by the element $w=[a_1,b_1]...[a_g,b_g]$. For every $n$ odd there exists a strongly irreducible quasi-isometric embedding $\tau_n:\Delta_n \rightarrow \mathsf{SL}(3n,\mathbb{R})$ which is not a limit of Anosov representations  and for every $\gamma \in \Gamma_n$, $\tau_n(\gamma)$ has all of its eigenvalues of maximum modulus real. \end{theorem}

\begin{proof} We first need to find a strongly irreducible representation of $\Delta_n$ into $\mathsf{SL}(n,\mathbb{R})$. We shall use a sequence of Johnson-Millson bending deformations defined in \cite{JM} to obtain such a representation. \par Let $\Delta=\Gamma_1\ast_{\langle a \rangle}\Gamma_2$ be the amalgamated free product of two torsion free word hyperbolic groups $\Gamma_1$ and $\Gamma_2$ along the maximal cyclic subroup generated by the element $a \in \Gamma_1 \cap \Gamma_2$. Suppose that  \hbox{$\rho: \Delta \rightarrow \mathsf{SO}(m,1)$} is a convex cocompact representation such that $\rho(\Gamma_{i})$ are Zariski dense in $\mathsf{SO}(m,1)$, and $\rho(a)$ lies in a copy of \hbox{$\mathsf{SO}_{0}{(2,1)} \subset \mathsf{SO}(m,1)$}. Then by the work of Johnson-Millson in \cite{JM} we can find a Zariski dense and convex cocompact deformation \hbox{$\rho_{t}: \Delta \rightarrow \mathsf{SO}(m+1,1)$} of $\rho$ (which we identify with $\textup{diag}(1,\rho)$). We briefly explain the construction (see also \cite[Lemma 6.3]{Kassel}): let $X$ be a vector in $\mathfrak{so}(m+1,1)-\mathfrak{so}(m,1)$ such that $\rho(w)X\rho(w)^{-1}=X$. Then, consider the family of representations $\rho_{t}:\Delta \rightarrow \mathsf{SO}(m+1,1)$ where $\rho_{t}(\gamma)=\rho(\gamma)$ for $\gamma \in \Gamma_1$ and $\rho_{t}(\gamma)=\exp(tX)\rho(\gamma)\exp(-tX)$ for $\gamma \in \Gamma_2$. Note that the Lie algebra $\mathfrak{so}(m,1)$ is self normalizing into $\mathfrak{so}(m+1,1)$. Therefore, for small enough $t>0$ the Lie algebra of the Zariski closure of $\rho_t$, $\mathfrak{g}_{t}$, strictly contains $\mathfrak{so}(m,1)$. It follows that $\rho_{t}$ is Zariski dense in $\mathsf{SO}(m+1,1)$. By the stability of convex cocompact representations into $\mathsf{SO}(m+1,1)$, established by Thurston \cite[Proposition 8.3.3]{Thurston} (see also \cite[Theorem 2.5.1]{CEG}), $\rho_t$ can be chosen to be convex cocompact.
\par Let $\big \{ \Gamma_i \big \}_{i=0}^{2^{n-3}}$ be the vertex groups of $\Delta_n$ and $\Delta_0$ be one of the surface vertex groups of $\Gamma_0$. By \cite{CMT} there exists a convex cocompact representation $\rho_{1}:\Delta_n \rightarrow \mathsf{SO}(3,1)$ such that $\rho_{1}|_{\Delta_0}$ is Fuchsian, i.e. $\rho_{1}|_{\Gamma_0}=S\circ \rho_0$ for some convex cocompact representation $\rho_{0}:\Delta_0 \rightarrow \mathsf{SL}(2,\mathbb{R})$. 
Notice that since $\Gamma$ is not a surface group, $\rho_1(\Gamma_i)$ has to be Zariski dense in $\mathsf{SO}(3,1)$. By the previous remarks we can find a convex cocompact representation $\rho_2:\Gamma_n \rightarrow \mathsf{SO}(4,1)$ such that for $0 \leqslant i \leqslant 2^{n-4}-1$, $\rho_{2}\big(\big \langle \Gamma_{2i+1}, \Gamma_{2i+2} \big \rangle \big)$ is Zariski dense in $\mathsf{SO}(4,1)$ and $\rho_2(\gamma)=\textup{diag}(1,\rho_1(\gamma))$ on $\gamma \in \Delta_0$. Now we see $\Delta_n$ as the amalgamated free product of $\Gamma_0$ with $\big \langle \Gamma_{1}, \Gamma_{2} \big \rangle,...,\big \langle \Gamma_{2^{n-3}-1}, \Gamma_{2^{n-3}} \big \rangle$ (each of them isomorphic to $\Gamma \ast_{\langle w \rangle} \Gamma$) along $\langle w \rangle$. Since for every $i$, $\rho_{2}(\langle \Gamma_{2i+1}, \Gamma_{2i+2}\rangle)$ is Zariski dense, we can find a convex cocompact representation $\rho_3:\Delta_n \rightarrow \mathsf{SO}(5,1)$ such that for $0 \leq i \leq 2^{n-5}-1$, $\rho_3\big(\big \langle \Gamma_{4i+1},\Gamma_{4i+2},\Gamma_{4i+3}, \Gamma_{4i+4}\big \rangle \big)$ is Zariski dense in $\mathsf{SO}(5,1)$ and $\rho_{3}(\gamma)=\textup{diag}\big(1,1,\rho_1(\gamma)\big)$. By continuing similarly we obtain a Zariski dense, convex cocompact representation $\rho_{n-3}:\Delta_n \rightarrow \mathsf{SO}(n-1,1)$ with $\rho_{n-3}(\gamma)=\textup{diag}\big (I_{n-4}, \rho_1(\gamma)\big)$ for $\gamma \in \Delta_0$.
\par Let $F$ be the free subgroup of $\Delta_0$ generated by the elements $a_1,b_1,...,a_g,b_g$ and let $R:\Delta_n \rightarrow F$ be a retract. We may choose a representation $A:F \rightarrow \mathsf{SL}(3,\mathbb{R})$ which is $P_1$-Anosov, $A\big(\langle a_3,b_3,...,a_g,b_g \rangle \big)$ is Zariski dense in $\mathsf{SL}(3,\mathbb{R})$ and $A(\gamma)=\textup{diag}(\rho_{0}(\gamma),1)$ for $\gamma \in \langle a_1,b_1, a_2,b_2 \rangle$. Now we consider $k$ very large and modify $A$ by considering \hbox{$A_{k}:F \rightarrow \mathsf{SL}(3,\mathbb{R})$} such that $A_{k}(\gamma)=\textup{diag}\big(\rho_{0}\big(\phi_{k}(\gamma)),1\big)$ for $\gamma \in \langle a_2,b_2 \rangle$ and $A_{k}(\gamma)=A(\gamma)$ for $\gamma \in \langle a_1,b_1,a_3,...,b_g \rangle$. The map \hbox{$\phi_{k}:\langle a_2,b_2 \rangle  \rightarrow \langle a_2,b_2 \rangle$} defined as in Observation \ref{free} and $k$ is considered large enough such that \hbox{$\ell_1(\rho_{0}(\phi_{k}(\gamma)) \geqslant \ell_{1}(\rho_{1}(\gamma))^{10}$} for every $\gamma \in \langle a_2,b_2 \rangle$. The image $A_{k}(F)$ defines a $P_1$-Anosov subgroup of $\mathsf{SL}(3,\mathbb{R})$.
\par Now we consider the representation $\tau_n:\Delta_n \rightarrow \mathsf{SL}(3n,\mathbb{R})$ defined as follows: $\tau_{n}(\gamma)=\rho_{n-3}(\gamma) \otimes A_{k}(R(\gamma))$ for $\gamma \in \Delta_n$. As in the proof of Theorem \ref{limit} (ii), $\tau_n$ is strongly irreducible since $\rho_{n-3}$ and $A\circ R$ have non-isomorphic irreducible Zariski closures. Moreover, all elements of the group $\tau_n(\Delta_n)$ have all of their eigenvalues of maximum modulus real, since $A|_{F}$ and $\rho_{n-3}$ are $P_1$-Anosov into $\mathsf{SL}(3,\mathbb{R})$ and $\mathsf{SL}(n,\mathbb{R})$ respectively. To see that $\tau_n$ is not a limit of Anosov representations we may first find $w \in (\Delta_{n})_{2}\cap \langle a_1,b_1\rangle $ such that $s:=\lambda_{1}(\rho_{0}(w))<0$. Then $g:=\rho_{n-3}(w)\otimes A_{k}(w)$ is conjugate to $\textup{diag}\big(s^2, I_{n-2},\frac{1}{s^2}\big)\otimes \textup{diag}\big(s,1,\frac{1}{s}\big)$. The first $2n-1$ eigenvalues of $g$ in decreasing order are $$s^3, s^2, \underbrace{s,...,s}_{n-1}, \underbrace{1,...,1}_{n-2}$$ and we see that $\wedge^{i}\tau_{n}(w)$ is not positively semiproximal when $i$ is even and $i \leqslant n+1$ and when $n+1 \leq i \leq 2n-1$. We may also find $w' \in (\Delta_n)_2 \cap \langle a_2,b_2 \rangle $ such that $q=\lambda_1(\rho_{0}(\phi_{k}(w'))<0$, let $p=\lambda_1(\rho(w'))$ and note that $|q|>p^{10}$. The matrix $h:=\rho_{n}'(w')\otimes A_{k}(w')$ is conjugate to the matrix $\textup{diag}\big(p^2,I_{n-2},\frac{1}{p^2}\big)\otimes \textup{diag}\big(q,1,\frac{1}{q}\big)$. The first $n+1$ eigenvalues of this matrix in decreasing order are: $$qp^2, \underbrace{q,q,...q}_{n-2}, \frac{q}{p^2}, p^2$$ The matrix $\wedge^{i}\tau_{n}(w')$ is not positively semiproximal when $i$ is odd and $i \leqslant n+1$. The conclusion follows by Lemma \ref{positive1}. \end{proof}

\medskip
In contrast to the previous examples, for the construction of quasi-isometric embeddings of surface groups which are not limits of Anosov representations we need to find elements whose eigenvalues are non-real. 

\medskip
\begin{proposition} \label{surface} Let $S$ be an orientable closed hyperbolic surface of genus at least $4$. There exist quasi-isometric embeddings $\psi:\pi_1(S) \rightarrow \mathsf{SL}(4,\mathbb{R})$ and $\rho:\pi_1(S) \rightarrow \mathsf{SL}(6,\mathbb{R})$ which are not limits of Anosov representations of $\pi_1(S)$ into $\mathsf{SL}(4,\mathbb{R})$ and $\mathsf{SL}(6,\mathbb{R})$ respectively. Moreover, $\rho$ is strongly irreducible. \end{proposition}

\begin{proof} Let $\rho_1:\pi_1(S)\rightarrow \mathsf{SL}(2,\mathbb{R})$ be a quasi-isometric embedding and $\pi:\pi_1(S) \rightarrow \langle a_1,a_2,a_3,a_4 \rangle$ be a retract of $\pi_1(S)$ onto the free subgroup $\langle a_1,a_2, a_3,a_4 \rangle$ of rank $4$. Let $\lambda:=\lambda_1(\rho_1(a_1))$ and $\mu:=\lambda_1(\rho_1(a_2))$ and fix $\theta \notin \pi \mathbb{Q}$.
\par We consider $x,y>0$ such that $x^2>|\lambda|$, $|\mu|>y^2$ and a homomorphism $\varepsilon: \langle a_1,a_2,a_3,a_4 \rangle \rightarrow \mathbb{R}^{+}$ with $\varepsilon(a_1)=x$ and $\varepsilon(a_2)=y$. Let $R_{\theta}:\langle a_1,a_2,a_3,a_4 \rangle \rightarrow \mathsf{SL}(2,\mathbb{R})$ be a homomorphism such that $R_{\theta}(a_1)$ and $R_{\theta}(a_2)$ are conjugate to the irrational rotation of angle $\theta$. The representation $\psi$ is defined as follows $$\psi(\gamma)=\begin{bmatrix}
\frac{1}{\varepsilon(\gamma)}\rho_1(\gamma) &0 \\ 
0 & \varepsilon(\gamma)R_{\theta}(\pi(\gamma))
\end{bmatrix}, \ \gamma \in \pi_1(S)$$ is a quasi-isometric embedding and not in the closure of Anosov representations. By the choice of $x,y$, the matrices $\psi(a)$ and $\wedge^2 \psi(b)$ have the numbers $xe^{i\theta}, xe^{-i \theta}$ and $\mu e^{i \theta}, \mu e^{-i \theta}$ as their eigenvalues of maximum modulus respectively. The claim follows.
\par  Now we construct the representation $\rho$. We consider $s,t, \theta \in \mathbb{R}$ satisfying $s>|\lambda|^{2/3}$, $|\mu|^{-2/3}<t<1$ and the representation $j_{s,t,\theta}:\langle a_1,a_2,a_3,a_4 \rangle \rightarrow \mathsf{SL}(3,\mathbb{R})$ such that $$j_{s,t,\theta}(a_1)=\begin{bmatrix}
s \cos\theta & -s \sin \theta  &0 \\ 
 s \sin \theta& s \cos \theta & 0\\ 
0 &0  & \frac{1}{s^2}
\end{bmatrix} \ \  j_{s,t,\theta}(a_2)=\begin{bmatrix}
t \cos\theta & -t \sin \theta  &0 \\ 
 t \sin \theta& t \cos \theta & 0\\ 
0 &0  & \frac{1}{t^2}
\end{bmatrix} $$ and $j_{s,t,\theta}(\langle a_3,a_4 \rangle)$ is Zariski dense in $\mathsf{SL}(3,\mathbb{R})$. By arguing as in Theorem \ref{limit}, the tensor product \hbox{$\rho:=\rho_1 \otimes (j_{s,t,\theta} \circ \pi)$} is a quasi-isometric embedding of $\pi_1(S)$ into $\mathsf{SL}(6,\mathbb{R})$ and also strongly irreducible. By the choice of $s$, the eigenvalues of the matrix $g:=\rho_1(a_1) \otimes j_{s,t,\theta}(\pi(a_1))$ in decreasing order of their moduli are $\lambda se^{i \theta}, \lambda se^{-i \theta}, \frac{s}{\lambda}e^{i \theta}, \frac{s}{\lambda}e^{-i\theta}, \frac{\lambda}{s^2}, \frac{1}{\lambda s^2}$. The matrices $g$ and $\wedge^3 g$ have their eigenvalues of maximum modulus non-real, hence $\rho_{\theta}$ is not a limit of $P_1$ or $P_3$-Anosov representations of $\pi_1(S)$ into $\mathsf{SL}(6,\mathbb{R})$. The eigenvalues of the matrix \hbox{$h:=\rho_1(a_2) \otimes j_{s,t,\theta}(\pi(a_2))$} in decreasing order of their moduli are $\frac{\mu}{t^2}, \mu t e^{i \theta}, \mu t e^{-i \theta}, \frac{1}{\mu t^2}, \frac{t}{\mu}e^{i \theta}, \frac{t}{\mu}e^{-i \theta}$. The matrix $\wedge^2 \rho_{\theta}(b)$ has its eigenvalues of maximum modulus non real, therefore $\wedge^2 \rho_{\theta}$ is not a limit of $P_1$-Anosov representations of $\pi_1(S)$. It follows that $\rho$ satisfies the required properties. \end{proof}

\begin{rmks} \textup{(i)} The construction in Proposition \ref{surface} works also for free groups.
\medskip

\noindent \textup{(ii)} We note that it is possible to describe the proximal limit set of the irreducible examples we have constructed. For a subgroup $H$ of $\mathsf{SL}(n,\mathbb{R})$, the proximal limit set $\Lambda^{\mathbb{P}}_H$ is defined to be the closure of the attracting fixed points of $P_1$-proximal elements of $H$ in $\mathbb{P}(\mathbb{R}^n)$. Let $\Delta$ be a non-elementary word hyperbolic group and suppose that \hbox{$\phi_1:\Delta \rightarrow \mathsf{SL}(n,\mathbb{R})$} and $\phi_2: \Delta \rightarrow \mathsf{SL}(m,\mathbb{R})$ are two irreducible representations such that $\phi_1\otimes \phi_2$ is irreducible, $\phi_1$ is \hbox{$P_1$-Anosov} and $\phi_2$ is either non-faithful or non-discrete. We claim that $\Lambda^{\mathbb{P}}_{(\phi_1 \otimes \phi_2)(\Delta)}$ is homeomorphic to $\Lambda^{\mathbb{P}}_{\phi_1(\Delta)}\times \Lambda^{\mathbb{P}}_{\phi_2(\Delta)}$. We may assume that $e_1 \otimes e_1$ is in $\Lambda^{\mathbb{P}}_{(\phi_1 \otimes \phi_2)(\Delta)}$ and $[e_1]\in \mathbb{P}(\mathbb{R}^n)$ is the attracting eigenline of $\phi_1(w_1)$. Let $w_0 \in \Delta$ be a non-trivial  element such that $\phi_2(w_0)=I_{m}$. For $x,y \in \partial_{\infty}\Delta$ with $\{x,y\}\cap \{w_{0}^{+},w_{0}^{-}, w_{1}^{+},w_{1}^{-}\}$ empty we may find a sequence $(\gamma_n)_{n \in \mathbb{N}}$ of elements of $\Delta$ with $x=\lim_{n}\gamma_n$ and $y=\lim_{n}\gamma_n^{-1}$. Then, $\lim_{n}(\gamma_n w_0 \gamma_{n}^{-1})w_{1}^{+}=x$ and hence $\lim_{n}(\phi_1\otimes \phi_2)(\gamma_n w_0\gamma_n^{-1})[e_1 \otimes e_1]=[u_x \otimes e_1]$ where $\xi_{\phi_1}^{1}(x)=[u_x]$. It follows that $\Lambda_{(\phi_1 \otimes \phi_2)(\Delta)}$ contains $[u_z \otimes e_1]$, $\xi_{\phi_1}^{1}(z)=[u_z]$, for every $z \in \partial_{\infty}\Delta$. Now since $(\phi_1 \otimes \phi_2)(\Delta)$ and $\phi_2(\Delta)$ act minimally on $\Lambda^{\mathbb{P}}_{(\phi_1 \otimes \phi_2)(\Delta)}$ and $\Lambda^{\mathbb{P}}_{\phi_2(\Delta)}$ respectively (see \cite[Lemma 2.5]{benoist-divisible0}), we conclude that \hbox{$\Lambda^{\mathbb{P}}_{(\phi_1 \otimes \phi_2)(\Delta)}=\big \{[u_1\otimes u_2]: [u_i] \in \Lambda^{\mathbb{P}}_{\phi_i(\Delta)}, \ i=1,2\big \}$}. We work similarly when $\phi_2$ is non-discrete. In particular, we deduce:
\medskip

\noindent \textup{(a)} In the constrution of $\rho$ in Theorem \ref{limit} (ii) the representation $A:F\rightarrow \mathsf{SL}(3,\mathbb{R})$ can be chosen to be non-discrete and hence the proximal limit set of $\rho(\Gamma)$ in $\mathbb{P}(\mathbb{R}^{12})$ is homeomorphic to $\partial_{\infty}\Gamma \times \mathbb{P}(\mathbb{R}^3)$.\\
\noindent \textup{(b)} In Theorem \ref{limit2}, the proximal limit set of $\tau_{n}(\Delta_n)$ in $\mathbb{P}(\mathbb{R}^{3n})$ is homeomorphic to $\partial_{\infty}\Delta_n \times C$, where $C$ is a Cantor set.\\
\textup{(c)} In Proposition \ref{surface}, for $s,t>0$ generic, the image of the representation $j_{s,t,\theta}$ is dense in $\mathsf{SL}(3,\mathbb{R})$ hence the proximal limit set of $\rho(\pi_1(S))$ in $\mathbb{P}(\mathbb{R}^{6})$ is homeomorphic to $S^1 \times \mathbb{P}(\mathbb{R}^3)$.\\

\section{Concluding remarks}
Let $G$ and $G'$ be two semisimple real algebraic Lie groups of real rank at least $2$ and $\iota: G \xhookrightarrow{} G'$ be an injective Lie group homomorphism. For an Anosov representation $\rho:\Gamma \rightarrow G$, the composition $\iota \circ \rho$ need not be Anosov into $G'$ with respect to any pair of opposite parabolic subgroups of $G'$. The failure of being Anosov under composing with a Lie group embedding has already been exhibited by Guichard-Wienhard, see the discussion in \cite[Section 4, p. 22]{GW}. Our examples are not limits of Anosov representations of their domain group into the bigger special linear group, but are Anosov when considered as representations into their Zariski closure. Let $P_{i}^{+}$ and $P_{i}^{-}$ be the subgroups of $\mathsf{SL}(d,\mathbb{R})$ defined as the stabilizers of a $i$-plane and a complementary $(d-i)$-plane of $\mathbb{R}^d$, $G=\mathsf{SL}(d,\mathbb{R})\times \mathsf{SL}(k,\mathbb{R})$ and denote by $\otimes:G \rightarrow \mathsf{SL}(kd,\mathbb{R})$ the tensor product representation. Let $\rho_1:\Gamma \rightarrow \mathsf{SL}(d,\mathbb{R})$ and $\rho_2:\Gamma \rightarrow \mathsf{SL}(k,\mathbb{R})$ be two representations such that $\rho_1$ is $P_i$-Anosov. The product $\rho_1 \times \rho_2:\Gamma \rightarrow G$ is Anosov with respect to the opposite parabolic subgroups $P_{i}^{+}\times \mathsf{SL}(k,\mathbb{R})$ and $P_{i}^{-}\times \mathsf{SL}(k,\mathbb{R})$ of $G$, while the tensor product $\rho_1 \otimes \rho_2:\Gamma \rightarrow \otimes(G)$ is Anosov with respect to the pair of opposite parabolic subgroups $\otimes(P_{i}^{+}\times \mathsf{SL}(k,\mathbb{R}))$ and $\otimes(P_{i}^{-}\times \mathsf{SL}(k,\mathbb{R}))$ of $\otimes(G)$. \par By following the lines of the proof of Theorem \ref{limit2}, we obtain examples of discrete and faithful representations which are not Anosov into their Zariski closure and not a limit of Anosov representations into $\mathsf{SL}(15,\mathbb{R})$. Let $M^3$ be a closed orientable hyperbolic $3$-manifold which is a surface bundle over the circle (with fibers homeomorphic to $S$) and contains a totally geodesic closed surface $S'$. By using the Klein combination theorem we may find a convex cocompact representation \hbox{$\rho:\pi_1(M^3)\ast F_2 \rightarrow \mathsf{SO}(4,1)$} whose restriction to the free factor $F_2$ is Zariski dense and \hbox{$\rho|_{\pi_1(M^3)}=\textup{diag}(1,\rho_{0})$} (here $\rho_0:\pi_1(M^3) \rightarrow \mathsf{SO}(3,1)$ denotes the holonomy representation associated to $M^3$). Since the quotient of $\pi_1(M^3)$ by the normal subgroup $\pi_1(S)$ is cyclic, the intersection \hbox{$H=\pi_1(S)\cap \pi_1(S')$} is a non-cyclic, normal free subgroup of $\pi_1(S')$. Let $F=\langle a_1,...,a_r \rangle$ be a free subgroup of $H$ of rank $r \geq 4$. We may find a finite cover $\hat{S}$ of $S$ such that $F\subset \pi_1(\hat{S})$ and there exists a retract $R:\pi_1(\hat{S}) \rightarrow F$ (see \cite[Theorem 1.6]{Long-Reid}) which we extend to a retract $R:\pi_1(\hat{S})\ast F_2 \rightarrow F$. Since $S'$ is totally geodesic in $M^3$ and $F$ is quasi-convex in $\pi_1(S')$, there exists a convex cocompact representation $\rho_1:F \rightarrow \mathsf{SL}(2,\mathbb{R})$ such that $\rho_0|_{F}=S\circ \rho_1$. As in Theorem \ref{limit2}, we consider $k$ very large and $A_k:F \rightarrow \mathsf{SL}(3,\mathbb{R})$ a Zariski dense representation, such that $A_{k}(\gamma)=\textup{diag}(1,\rho_1(\gamma))$ for $\gamma \in \langle a_1,a_2 \rangle$ and $A_{k}(\gamma)=\textup{diag}(1,\rho_1(\phi_k(\gamma)))$ for $\gamma \in \langle a_3,a_4 \rangle$. The representation $\rho':\pi_1(\hat{S})\ast F_2 \rightarrow \mathsf{SL}(15,\mathbb{R})$, $\rho'(\gamma)=\rho(\gamma)\otimes A_k(R(\gamma))$ for $\gamma \in \pi_1(\hat{S})\ast F_2$, is discrete and faithful (since $\rho$ is) and not in the closure of Anosov representations of $\pi_1(\hat{S})\ast F_2$ into $\mathsf{SL}(15,\mathbb{R})$. We note that since $A_k \circ R$ is not faithful and $\pi_1(S)$ is normal in $\pi_1(M^3)$, the representations $A_k \circ R$ and $\rho|_{\pi_1(\hat{S})}$ are not quasi-isometric embedings into $\mathsf{SO}(4,1)$ and $\mathsf{SL}(3,\mathbb{R})$ respectively. In particular, $\rho \times (A_k \circ R)$ is not Anosov with respect to any pair of opposite parabolic subgroups of $\mathsf{SO}(4,1)\times \mathsf{SL}(3,\mathbb{R})$. The Zariski closure of $\rho'$ is $\otimes(\mathsf{SO}(4,1)\times \mathsf{SL}(3,\mathbb{R}))$ and it follows (see for example \cite[Corollary 3.6]{GW}) that $\rho \otimes (A_k \circ R)$ is not Anosov in its Zariski closure.
 \end{rmks}

\end{document}